\documentclass{article}
\usepackage{amsmath, amscd, amsfonts, amsthm, amssymb, enumerate, hyperref}
\usepackage[all]{xy}

\newtheorem{thm}{Theorem}[section]
\newtheorem{lem}[thm]{Lemma}
\newtheorem{prop}[thm]{Proposition}
\newtheorem{cor}[thm]{Corollary}
\theoremstyle{definition}

\newtheorem{defn}[thm]{Definition}

\theoremstyle{remark}
\newtheorem*{rem}{Remark}

\DeclareMathOperator{\tr}{tr}
\DeclareMathOperator{\cro}{cr}
\DeclareMathOperator{\lcm}{lcm}
\DeclareMathOperator{\sgn}{sgn}
\newcommand{\PGLQ}{\mathrm{PGL}_2\mathbb{Q}}
\newcommand{\PGLZ}{\mathrm{PGL}_2\mathbb{Z}}
\newcommand{\GLQ}{\mathrm{GL}_2\mathbb{Q}}
\newcommand{\GLZ}{\mathrm{GL}_2\mathbb{Z}}
\newcommand{\NN}{\mathbb{N}}
\newcommand{\QQ}{\mathbb{Q}}
\newcommand{\PP}{\mathbb{P}}
\newcommand{\RR}{\mathbb{R}}
\newcommand{\ZZ}{\mathbb{Z}}
\newcommand{\OO}{\mathcal{O}}
\newcommand{\xih}{\hat{\xi}}
\newcommand{\de}{\delta_\epsilon}
\newcommand{\textand}{\quad \text{and} \quad}

\title{Continued Fractions\\and Linear Fractional Transformations}
\author{Evan O'Dorney}
\date{\today}
\begin{document}
\maketitle

\begin{abstract}
Rational approximations to a square root $\sqrt{k}$ can be produced by iterating the transformation $f(x) = (dx+k)/(x+d)$ starting from $\infty$ for any $d \in \NN$. We show that these approximations coincide infinitely often with continued fraction convergents if and only if $R = 4d^2/(k-d^2)$ is an integer, in which case the continued fraction has a rich structure. It consists of the concatenation of the continued fractions of certain explicitly definable rational numbers, and it belongs to one of infinitely many families of continued fractions whose terms vary linearly in two parameters. We also give conditions under which the orbit $\{f^n(\infty)\}$ consists exclusively of convergents or semiconvergents and prove that with few exceptions it includes all solutions $p/q$ to the Pell equation $p^2 - k q^2 = \pm 1$.
\end{abstract}

\section{Introduction}

Let $k$ be a fixed non-square positive integer. Among the simplest of dynamical systems one can use to approximate the irrational number $\sqrt{k}$ is the family of linear fractional transformations (hereafter LFT's)
\begin{equation}\label{eq:f}
  f(x) = f_d(x) = \frac{d x + k}{x + d}
\end{equation}
on $\RR\PP^1$. If $d > 0$, then $f$ has $\sqrt{k}$ as its unique attracting fixed point, and thus the iterates $\{f^n(\infty)\}_{n=1}^{\infty} = \{f(\infty), f(f(\infty)), \ldots\}$ form a sequence of approximations converging to $\sqrt{k}$, which are rational if $d$ is. Such a procedure is not new. Theon of Smyrna was iterating the LFT $x \mapsto \frac{x+2}{x+1}$ to approximate $\sqrt{2}$ as early as the second century AD \cite{Cur}. Also, if $d = p/q$ is \emph{Pellian}, that is, satisfies the Pell equation $p^2 - kq^2 = \pm 1$, then the iterates of $f$ correspond to the powers of $p + q\sqrt{k}$ (see Lemma \ref{lem:AB}) used to produce further solutions to the Pell equation, which can then be used to solve more involved quadratic Diophantine problems such as Archimedes' cattle problem \cite{Vardi}.

Finally, we cannot leave out the connection with a much faster-converging dynamical system commonly used in electronic square-root algorithms. It is sometimes known as the ``Babylonian Method''; it is also equivalent to Newton's method applied to the equation $x^2 - k = 0$. It is the following easily discovered nonlinear transformation:
\[
  F(x) = \frac{1}{2}\left(x + \frac{k}{x}\right).
\]
It is not hard to verify that if $(a + b\sqrt{k})^2 = c + d\sqrt{k}$ ($a,b,c,d \in \QQ$), then $F(a/b) = c/d$, so the sequence of iterates of $F$ on a seed value $d$ consists of the iterates
\[
  d = f_d(\infty), \quad f_d^2(\infty), \quad f_d^4(\infty), \quad f_d^8(\infty), \ldots
\]
of the corresponding $f_d$.

Our concern in this paper is how well the iterates $f^i(\infty)$ approximate $\sqrt{k}$ in comparison to its canonical sequence of best possible approximations, the continued fraction convergents $\{p_n/q_n\}$ and especially those that are Pellian. In \cite{RST}, J. Rosen, K. Shankar, and J. Thomas considered the case $d = \lfloor \sqrt{k} \rfloor$ and proved that the orbit of $f$ coincides with the sequence of convergents if and only if the continued fraction has period $1$ or $2$, which in turn is equivalent to the condition that $\frac{2d}{k-d^2}$ is an integer. Building upon this line of reasoning, S. Mikkilineni proved in \cite{Mikki} that if $\frac{2d}{k-d^2} = \frac{m}{2}$, where $m \geq 3$ is an odd integer, then the sequence $\{f^n(\infty)\}_{n=0}^{\infty}$ forms a subsequence $\{p_{j_n}/q_{j_n}\}$ of the sequence of convergents, with the indices $j_n$ depending only on the parity of $k$. In particular, this subsequence always contains every Pellian convergent.

Mikkilineni conjectured that similar behavior occurs when $\frac{2d}{k-d^2} = \frac{m}{2^h}$ for any $h \geq 1$, provided that $2^{h-1}|d$ and certain mild inequalities hold (\cite{Mikki}, Conjecture 4.5). Specifically she conjectured that the orbit of $\infty$ forms a subsequence of the sequence of convergents; that this subsequence is invariant with respect to a certain parameter $m$; and that moreover the continued fraction has period independent of $m$ and terms linear in $m$. Unfortunately, her method of proof for $h=1$---computing all the terms and convergents of one period of the continued fraction---is not well suited for proving her conjectural generalization, as the periods of the continued fractions occurring in it can be arbitrarily long.

Our results prove Mikkilineni's conjectures and extend them in several different directions. First, we show that the relevance of $f$ to the continued fraction is encapsulated nicely by the quantity
\[
  R = \frac{4d^2}{k-d^2} :
\]
only finitely many iterates $f^n(\infty)$ are convergents unless $R \in \ZZ$, in which case the iterates include infinitely many Pellian convergents (Theorem \ref{thm:easy}). If, in addition, $d$ is the nearest integer to $\sqrt{k}$, then the orbit consists entirely of convergents of the continued fraction, a fact (Theorem \ref{thm:conv}) for which we give two proofs. The first uses clever manipulation of inequalities; the second deduces it as a corollary (Theorem \ref{thm:conv2}) of a striking result (Theorem \ref{thm:patterns}) that allows the continued fraction expansion of $\sqrt{k}$ to be computed as a concatenation of certain finite continued fractions. Since these rational numbers are defined in terms of division in modular arithmetics, which is itself usually computed by means of continued fractions, one can say that we have a continued fraction within a continued fraction. As a by-product, we get families of continued fractions whose terms are bilinear in two free parameters (Theorem \ref{thm:family}). Lastly, we show that in almost all cases, $R \in \ZZ$ implies that the orbit includes \emph{all} Pellian convergents (Theorem \ref{thm:pell}).

In the course of the development, it will become increasingly clear that the natural surds to study are not square roots of integers but general real algebraic integers of degree $2$: to cite the extreme case, the proof of Theorem \ref{thm:pell} rests on a kind of induction in which some of the $\sqrt{k}$ cases are reduced to $\frac{1 + \sqrt{4m+1}}{2}$ for integers $m$. Because the $\sqrt{k}$ case is of general interest, however, we follow several of our theorems with corollaries spelling out the results that they yield in this case.

\section{Conditions for connections between $f$ and the continued fraction}
\label{sec:prelim}

First, we reinterpret $f$ in a simple way.

\begin{lem}\label{lem:AB}
Let $k$ and $d$ be positive rational numbers with $k$ not a square. Then for each $n \geq 0$, we have $f^n(\infty) = a/b$, where $a$ and $b$ are the rational numbers satisfying
\[
  (d + \sqrt{k})^n = a + b\sqrt{k}.
\]
\end{lem}
\begin{proof}
Induction. The $n=0$ case is trivial, and if $f^n(\infty) = a/b$, then
\[
  f^{n+1}(\infty) = f\left(\frac{a}{b}\right) = \frac{da + kb}{a + db}
\]
and
\[
  (a + b \sqrt{k})(d + \sqrt{k}) = da + kb + (a + db)\sqrt{k}. \qedhere
\]
\end{proof}
\begin{thm} \label{thm:easy}
Let $k$ be a non-square positive integer, and let $d$ be a positive rational number. The following are equivalent:
\begin{enumerate}[$(a)$]
  \item $R = \frac{4d^2}{k-d^2}$ is an integer.
  \item Some iterate is a Pellian convergent of $\sqrt{k}$.
  \item Infinitely many of the iterates are continued fraction convergents of $\sqrt{k}$.
\end{enumerate}
\end{thm}

\begin{proof}
We first prove that (b) is equivalent to (c). If (b) holds, then
\[
  (d + \sqrt{k})^n = r(p + q\sqrt{k}),
\]
where $r$ is rational and $p$ and $q$ are integers satisfying the Pell equation $p^2 - k q^2 = \pm 1$. Then for all $i \geq 1$, we have
\[
  (d + \sqrt{k})^{in} = r^i(P + Q\sqrt{k}),
\]
where (due to the multiplicativity of the norm) $(P,Q)$ also satisfies the Pell equation. Then $f^{in}(\infty) = P/Q$ is Pellian and thus a continued fraction convergent.

Conversely, assume that $\{f^n(\infty)\}$ includes infinitely many convergents $p_j/q_j = [c_0,\ldots,c_{j-1}]$ of the continued fraction $\sqrt{k} = [c_0,c_1,\ldots]$. Then, by the pigeonhole principle, two of these have indices that are congruent mod $L$, the period of the continued fraction. Suppose that
\[
  f^{n_1}(\infty) = \frac{p_j}{q_j} \textand
  f^{n_2}(\infty) = \frac{p_{j+\ell L}}{q_{j+\ell L}}.
\]
Since $\sqrt{k} = [c_0,\overline{c_1,\ldots,c_L}]$ is almost purely periodic, there is an LFT $g(x) = [c_0,c_1,\ldots,c_{L-1},c_L-c_0+x]$ that fixes $\sqrt{k}$ (and, by rationality, $-\sqrt{k}$) and advances each convergent to the $L$th succeeding one. Note that $f^{n_2-n_1}$ and $g^\ell$ take the same value at the three points $\sqrt{k}$, $-\sqrt{k}$, and $p_j/q_j$; thus they are equal. In particular $f^{n_2 - n_1}(\infty) = g^\ell(\infty) = p_{\ell L}/q_{\ell L}$ is Pellian.

To connect (a) and (b), we use the matrix interpretation of transformations in $\PGLQ$.

\begin{lem}\label{lem:trace}
  Let $A \in \GLQ$ be a matrix representing a transformation in $\PGLQ$. The following conditions are equivalent:
\begin{enumerate}[$(a)$]
\item $\dfrac{(\tr A)^2}{\det A}$ is an integer.
\item There exists $n \in \mathbb{N}$ such that $A^n$ represents a transformation in $\PGLZ$; that is, $A^n = r B$ where $r \in \QQ$ and $B \in \GLZ$.
\end{enumerate}
\end{lem}
\begin{proof}
It is obvious that condition (b) is unchanged if $A$ is replaced by $A^2$. Let us prove that (a) has the same property. If the eigenvalues of $A$ are $\lambda_1, \lambda_2 \in \mathbb{C}$, then $\tr A = \lambda_1 + \lambda_2$, $\det A = \lambda_1 \lambda_2$, and
\[
	\tr A^2 = \lambda_1^2 + \lambda_2^2 = (\lambda_1 + \lambda_2)^2 - 2 \lambda_1 \lambda_2
	= (\tr A)^2 - 2 \det A.
\]
Consequently
\[
	\frac{(\tr A^2)^2}{\det A^2} = \left(\frac{(\tr A)^2 - 2 \det A}{\det A}\right)^2 =
	\left(\frac{(\tr A)^2}{\det A} - 2 \right)^2.
\]
The claim now follows from the fact that if $x \in \QQ$ and $(x-2)^2 \in \ZZ$ then $x \in \ZZ$.

Thus we may restrict to the case where $A$ is the square of a matrix in $\PGLQ$. In particular, we may assume that $\det A$ is a square in $\mathbb{Q}$, which, by scaling, we may take to be $1$. Then condition (a) becomes the statement that $\tr A$ (which we will denote by $t$) is an integer. Condition (b), since $A^n$ already has determinant $1$, becomes the condition that $A^n$ has integer entries for some $n \in \mathbb{N}$.

Let us prove that (b) implies (a). The eigenvalues of $A$ have product $1$; denote them by $\lambda$ and $1/\lambda$. Suppose that $A^n$ has integer entries and let $T = \tr A^n$. Then $\lambda^n + 1/\lambda^n = T$, or $\lambda^{2n} - T \lambda^n + 1 = 0$. This implies that $\lambda$ is an algebraic integer, and symmetrically we know that $1/\lambda$ is an algebraic integer. Thus $t = \lambda + 1/\lambda$ is an algebraic integer. Since $t$ is rational, $t$ must be an integer.

Now let us assume (a), that $t$ is an integer, and prove (b). Let $m$ be a common denominator for the entries of $A$, i.e$.$ a nonzero integer such that $mA$ has integer entries. By the Cayley-Hamilton theorem,
\[
	mA^{n+1} - tmA^n + mA^{n-1} = 0,
\]
from which we see that $mA^n$ has integer entries for all $n \geq 0$. Let $x_{i,j}(n)$ ($i, j \in \{1, 2\}$, $n \geq 0$) denote the $(i,j)$ entry of $mA^n$. We have the linear recurrence
\[
	x_{i,j}(n+1) - t x_{i,j}(n) + x_{i,j}(n-1) = 0.
\]
Mod $m$, the sequence $x_{i,j}(n)$ for fixed $i,j$ must be purely periodic (this is a general property of linear recursive sequences whose leading and trailing coefficients are relatively prime to the modulus). By taking the LCM over the four possible combinations $(i,j)$, we find that there is a period $\ell$ with respect to which all four sequences are periodic. We know that $x_{i,j}(0) \equiv 0$ mod $m$; it follows that $x_{i,j}(\ell) \equiv 0$ mod $m$, that is, that $A^\ell$ has integer entries.
\end{proof}
To prove the theorem, take $A = \begin{bmatrix} d & k \\ 1 & d \end{bmatrix}$. Condition (a) of the lemma is clearly equivalent to (a) of the theorem. If (b) of the theorem holds, then for some $n$, $f^n$ is the unique transformation
\[
  x \mapsto \frac{px + kq}{qx + p}
\]
fixing $\pm \sqrt{k}$ and taking $\infty$ to the Pellian convergent $p/q$. Then
\begin{equation}\label{eq:pellmatrix}
	A^n = r \begin{bmatrix} p & kq \\ q & p \end{bmatrix}
\end{equation}
for some $r \in \QQ$. Since $\det A^n/r = p^2 - kq^2 = \pm 1$, we have (b) of the lemma. Conversely, if (b) of the lemma holds, then \eqref{eq:pellmatrix} holds for some $r \in \QQ$ and $p,q \in \ZZ$ satisfying $p^2 - kq^2 = \pm1$; thus $f^n(\infty) = p/q$ is Pellian.
\end{proof}

Under certain conditions, the other iterates of $f$ bear a significant relationship to the continued fraction as well.
\begin{thm} \label{thm:conv}
If $d$ is the nearest integer to $\sqrt{k}$ (that is, $d = \lfloor k + 1/2 \rfloor$) and $R \in \ZZ$, then the iterates of $f$ on $\infty$ are all convergents of the continued fraction for $\sqrt{k}$.
\end{thm}
\begin{proof}
The proof will proceed in the following steps:
\begin{enumerate}[(1)]
\item We will prove that any iterate of $f$ has the form $p/q$ where $|p^2 - k q^2| \leq |k - d^2|$;
  \item We will prove that $|k - d^2| < \sqrt{k}$;
  \item We will appeal to a well-known theorem that if $p/q$ is a positive fraction satisfying $|p^2 - k q^2| < \sqrt{k}$, then $p/q$ is a convergent of $\sqrt{k}$.
\end{enumerate}

To prove step (1), let $K$ be the number field $\QQ[\sqrt{k}]$; let $\OO_K$ be its ring of integers, and let $\OO$ be the order $\ZZ[\sqrt{k}] \subseteq \OO_K$. Note that
\begin{align*}
  \zeta &= \frac{\left(d + \sqrt{k}\right)^2}{d^2 - k} = \frac{R + 2 + \sqrt{R(R+4)}}{2}
\end{align*}
is a unit in $\OO_K$ (indeed, it satisfies the equation $\zeta^2 - (R+2)\zeta + 1 = 0$). If $n = 2i+1$ is odd, then
\[
  \frac{(d+\sqrt{k})^n}{(d^2-k)^i} = (d + \sqrt{k}) \zeta^i \in \OO \cdot \OO_K = \OO
\]
is an element $p + q\sqrt{k} \in \OO$ with norm $p^2 - k q^2 = d^2 - k$; and we have $f^n = p/q$ by Lemma \ref{lem:AB}.

If $n = 2i$ is even, then we must look at $\zeta^i$ instead, and we seek an $s \in \NN$ such that $s\zeta^i \in \OO$. It is evident that any $s$ that works for $i=1$ will work for all $i$. So we want $s\zeta \in \OO$, which holds if and only if the irrational part
\[
  s \cdot \frac{2d\sqrt{k}}{d^2 - k}
\]
of $s\zeta$ is a multiple of $\sqrt{k}$, from which it follows that the minimal $s$ is
\begin{align*}
  s &= \lcm\left(\frac{|k - d^2|}{2d}, 1\right)
\end{align*}
and so (letting $S = |k - d^2|$ for brevity)
\begin{align*}
  s^2 &= \lcm\left(\frac{S^2}{4d^2}, 1\right) \\
  & \leq \lcm\left(\frac{S^2}{4d^2}, S\right) \\
  & = \lcm\left(\frac{S}{R},S\right) \\
  & = S,
\end{align*}
as desired.

Next, we prove that $|d^2 - k| < \sqrt{k}$ by arguing that
\[
  |d^2 - k| = |d - \sqrt{k}| (d + \sqrt{k}) < \frac{d + \sqrt{k}}{2}.
\]
Since $d$ is the nearest integer to $\sqrt{k}$, the average $(d + \sqrt{k})/2$ sits between the same two consecutive integers as $\sqrt{k}$ does, and hence the integer $|d^2 - k|$ is less than $(d + \sqrt{k})/2$ if and only if it is less than $\sqrt{k}$.

To finish the proof, we appeal to the following well-known result.
\end{proof}

\begin{lem}\label{lem:conv}
If $p$ and $q$ are positive integers such that $|p^2 - k q^2| < \sqrt{k}$, then $p/q$ is a convergent of $\sqrt{k}$.
\end{lem}
\begin{proof}
See \cite{NZM}, Theorem 7.24.
\end{proof}

A \emph{semiconvergent} of a continued fraction $[c_0, c_1, \ldots]$ is an approximation $[c_0, \ldots, c_{n-1}, b_n]$ where $0 \leq b_n \leq c_n$. By various measures the semiconvergents are the next best approximations after the convergents (see \cite{NZM}, Exercise 7.5). The following theorem can be proved analogously to the preceding; but since it will be deduced from the methods in the next section, the proof is left as an exercise for the interested reader.

\begin{thm} \label{thm:semi}
If $d$ is one of the two nearest integers to $\sqrt{k}$ (that is, $d = \lfloor k \rfloor$ or $d = \lceil k \rceil$) and $R \in \ZZ$, then the iterates of $f$ on $\infty$ are all semiconvergents of the continued fraction for $\sqrt{k}$.
\end{thm}

\section{Patterns in the continued fraction}
\label{sec:patterns}

We now proceed to compute the continued fraction explicitly. We begin by parametrizing the admissible values of $k$ and $d$.
\begin{prop}
\label{prop:parametrize}
If $k$ and $d$ are positive integers such that $R \in \ZZ$, then there are positive integers $s$, $v$, and $m$ such that
\[
  d = \frac{svm}{2} \textand k = \frac{s^2 v (vm^2 + 4\epsilon)}{4},
\]
where $\epsilon = \sgn (k-d^2) = \pm 1$.
\end{prop}
\begin{proof}
Let $v = \gcd(R, k-d^2)$. Note that $|R|/v$ and $|k-d^2|/v$ are two relatively prime positive integers whose product is $4d^2/v^2$, a square. Therefore $|R|/v = m^2$ and $|k-d^2|/v = s^2$ for positive integers $m$ and $s$. Then $ms = 4d/v$, giving us $d = svm/2$ and
\[
  k = d^2 + \epsilon s^2v = \frac{s^2 v^2 m^2}{4} + \epsilon s^2 v = \frac{s^2 v (vm^2 + 4\epsilon)}{4}. \qedhere
\]
\end{proof}

In the following we will be less interested in $\sqrt{k}$ itself as in the number $\xi = d + \sqrt{k}$, an algebraic integer satisfying the equation
\[
  \xi^2 - svm \xi - \epsilon s^2 v = 0
\]
which is a fixed point of the LFT
\[
  f_\xi(x) = f(x - d) + d = 2d + \frac{k - d^2}{x} = svm - \frac{\epsilon s^2 v}{x}.
\]
As long as $d$ is an integer, this is a harmless shift of $f$; when $s$, $v$, and $m$ are all odd, we have made a slight generalization (for instance, $s=v=m=1$ yields the LFT $f_\xi(x) = 1 + 1/x$ fixing the golden ratio $\xi = \frac{1+\sqrt{5}}{2}$). The only $(s,v,m,\epsilon)$ quadruples we have to exclude are those where $\xi$ is rational or non-real, which happens only in the case that $\epsilon = -1$ and $vm^2 \leq 4$.

It will be useful to introduce the notation
\[
  \de = \begin{cases} 0 & \text{if $\epsilon = 1$} \\ 1 & \text{if $\epsilon = -1$.} \end{cases}
\]

We now introduce the quantities in terms of which we will express the continued fraction.

\begin{prop}\label{prop:a}
Let
\[
  v_n = \begin{cases} v & \text{if $n$ is even} \\ 1 & \text{if $n$ is odd} \end{cases}
\]
and consider the sequence $\{a_n\}$ of integers defined by
\[
  a_0 = 0, \quad a_1 = 1, \quad a_{n+1} = v_n m a_n + \epsilon a_{n-1}.
\]
Then
\begin{enumerate}[$(a)$]
  \item $\gcd(a_n,a_{n+1}) = 1;$
  \item $f_\xi^n(\infty) = s v_{n+1} a_{n+1}/a_n$ (thus $a_n > 0$ for $n \geq 1$).
\end{enumerate}
\end{prop}
\begin{proof}
Simple inductions.
\end{proof}

It is to be noted that the sequence $a_n$ is easily computable using either its defining recursion or the explicit formula
\[
  a_n = \frac{(d+\sqrt{k})^n - (d-\sqrt{k})^n}{2s^{n-1} v^{\left\lfloor \frac{n}{2} \right\rfloor}\sqrt{k}}.
\]

The essence of the following theorem is that the continued fraction expansion of $\xi$ consists of a string of ``packets,'' each of which corresponds to the reduction of $-a_{n-1}/a_n = -s v_{n+1} / f_\xi^n(\infty)$ modulo $s$. But since $s$ and $a_n$ may share factors, we must instead reduce $[-a_{n-1}:a_n]$ to a point of the projective line $\PP^1(\ZZ/s\ZZ)$, which can be specified by two numbers: a divisor $s_n$ of $s$ and a congruence class $m_n$ modulo $s_n$ such that the equality of points
\[
  [-a_{n-1}:a_n] = [m_n:s_n]
\]
holds in $\PP^1(\ZZ/s\ZZ)$. We note that the sequence $a_n$ can be extended to negative $n$; in particular, it is purely periodic to any finite modulus.

One more remark is in order before the theorem is stated. As is well known, any rational number has two finite simple continued fraction expansions (allowing for a nonpositive initial term), because
\[
  [c_0, \ldots, c_{n-1}, c_n, 1] = [c_0, \ldots, c_{n-1}, c_n + 1].
\]
Their lengths differ by exactly $1$. For most applications, the shorter expansion is preferred; but here we find it necessary to select one or the other based on the \emph{parity} of their lengths.

\begin{thm} \label{thm:patterns}
Let $s$, $v$, and $m$ be positive integers, and let $\epsilon = \pm 1$. Define $d$, $k$, and $\xi = d + \sqrt{k}$ according to the formulas in Proposition \ref{prop:parametrize}, and let the sequences $\{v_n\}$ and $\{a_n\}$ be as in Proposition \ref{prop:a}. Furthermore define
\begin{align*}
  s_n &= \gcd(a_n,s) \\
  m_n &= \left( -\frac{a_{n-1}}{(a_n/s_n)} \right) \bmod \frac{s}{s_n},
\end{align*}
where the last equation means to perform the division mod $s/s_n$ (which is possible since $\gcd(a_n/s_n,$
$s/s_n) = 1$) and express the result as an integer $m_n$, with $0 \leq m_n < s/s_n$. Let
\[
  \xih_n = \frac{s_n v_n m - \epsilon m_n}{s/s_n}.
\]
Then:
\begin{itemize}
  \item If $\epsilon = 1$, then $\xi$ has a continued fraction expansion formed by concatenating those of $\xih_0, \xih_1, \xih_2,\ldots,$ when these are chosen to have an odd number of terms.
  \item If $\epsilon = -1$, then $\xi$ has a continued fraction expansion formed by concatenating those of $\xih_0, \xih_1-1, \xih_2-1, \xih_3-1, \ldots,$ when these are chosen to have an even number of terms.
  \item In either case, the convergent formed by the first $n$ of these finite continued fractions is the approximant $f_\xi^n(\infty)$.
\end{itemize}
\end{thm}

To avoid confusion, we call the continued fraction expansion of $\xi$ produced by this theorem the \emph{pattern continued fraction,} to be distinguished from the simple continued fraction expansion which only coincides with it when all $\xih_n$ are at least $1$ (for $\epsilon = 1$) or greater than $2$ (for $\epsilon = -1$).

It is to be noted that the quantities $s_n$ and $m_n$, hence $\xih_n$, depend only on the sequences $\{v_n\}$ and $\{a_n\}$ mod $s$ and therefore are purely periodic. So we obtain a purely periodic continued fraction expansion for $\xi - \de$. If it happens to be simple, then by the well-known criterion for pure periodicity, the Galois conjugate $\bar{\xi}$ lies between $-1$ and $0$; appropriate converses to this will soon be proved (Theorems \ref{thm:conv2} and \ref{thm:semi2}).

We first state and prove a lemma that is useful in general when continued fractions are being concatenated.
\begin{lem}\label{lem:concat}
Let $p/q$ be a rational number in lowest terms $(q > 0)$ and let $[c_0,c_1,\ldots,c_n]$ be either of its simple continued fraction expansions, where $c_0 \in \ZZ$ and all other terms are positive. Then we have the equality of LFT's
\[
  [c_0,c_1,\ldots,c_n,x] = \frac{px + g}{qx + h}
\]
where $(g,h)$ is the unique solution to $gq - hp = (-1)^n$ satisfying $0 \leq h < q$ (if $n$ is even) or $0 < h \leq q$ (if $n$ is odd).
\end{lem}
\begin{proof}
Let $\tau(x) = [c_0,c_1,\ldots,c_n,x]$. Since $\tau(\infty) = p/q$, we have
\[
  \tau(x) = \frac{px + g}{qx + h}
\]
for some integers $g$ and $h$; since $x$ is reciprocated $n+1$ times, the determinant $h p - g q$ is $(-1)^{n+1}$. Note that $\tau(x)$ has a finite value whenever $x$ is greater than $0$, less than $-1$, or equal to $\infty$; so the unique pole of $\tau$ lies between $-1$ and $0$ inclusive, and thus $0 \leq h \leq q$. Since the determinant condition has a unique solution mod $q$, we have determined $h$ uniquely unless $h \equiv 0$ mod $q$, which can only happen if $q = 1$ and $p/q = p$ is an integer. Here the relevant continued fractions are
\[
  [p,x] = \frac{px + 1}{x} \textand [p-1,1,x] = \frac{px+(p-1)}{x+1},
\]
with $h$ respectively taking the values $0$ and $1$. These cases can be told apart by the parity of $n$ as in the statement of the lemma.
\end{proof}

We now proceed with the proof of Theorem \ref{thm:patterns}.
\begin{lem}
Let
\[
  \xi_n = \frac{\xi s_n^2 v_n - \epsilon s s_n v m_n}{s^2 v} = \xih_n + \frac{\bar{\xi} s_n^2 v_n}{s^2 v},
\]
where $\bar{\xi} = d - \sqrt{k}$ is the $\QQ$-Galois conjugate of $\xi$. If the continued fraction expansion of $\xih_n - \de$ is $[c_0,\ldots,c_n]$, where $n \equiv \de$ mod $2$, then
\[
  \xi_n - \de = [c_0,\ldots,c_n,\xi_{n+1} - \de].
\]
\end{lem}
\begin{proof}
It is of course equivalent to prove that
\[
  \xi_n = [c'_0, c_1, \ldots, c_n,\xi_{n+1} - \de]
\]
where $[c'_0, c_1, \ldots, c_n]$ is the continued fraction expansion of $\xih_n$ itself.

First, let us calculate the greatest common divisor $w$ of the numerator and denominator of
\[
  \xih_n = \frac{s_n v_n - \epsilon m_n}{s/s_n}
\]
by expressing the congruence class of the numerator $u = s_n v_n - \epsilon m_n$ modulo $s/s_n$ in a simple way:
\begin{align*}
  u &= s_n v_n - \epsilon m_n \\
  &\equiv s_n v_n + \epsilon \frac{a_{n-1}}{(a_n/s_n)} \\
  &\equiv \frac{a_n v_n + \epsilon a_{n-1}}{a_n/s_n} \\
  &\equiv \frac{a_{n+1}}{a_n/s_n} \mod \frac{s}{s_n}.
\end{align*}
In particular, since $s_n$ and $a_{n+1}$ are relatively prime (by Proposition \ref{prop:a}(a)), we have
\[
  w = \gcd\left(u,\frac{s}{s_n}\right) = \gcd\left(a_{n+1},\frac{s}{s_n}\right) = \gcd(a_{n+1},s) = s_{n+1},
\]
so the numerator and denominator of $\xih_n$ are
\[
  p = \frac{u}{s_{n+1}} \textand q = \frac{s}{s_n s_{n+1}}.
\]
Now let the continued fraction expansion of $\xih = p/q$ be $[c'_0,c_1,\ldots,c_n]$, where $n \equiv \de$ mod $2$. Applying Lemma \ref{lem:concat}, we have
\[
  [c'_0,c_1,\ldots,c_n,x] = \frac{p x + g}{q x + h}
\]
where $g$ and $h$ are determined by the relations $g q - h p = \epsilon$ and $0 \leq h < q$ (for $\epsilon = 1$) or $0 < h \leq q$ (for $\epsilon = -1$). We would like to prove the relation
\[
  \xi_n = \frac{p\xi_{n+1} + g}{q\xi_{n+1} + h},
\]
or equivalently
\[
  \xi_{n+1} = \frac{g - h \xi_n}{q \xi_n - p}.
\]
Recall that $\xi_n = \xih_n + \kappa = p/q + \kappa$, where $\kappa = \bar{\xi}s_n^2 v_n/s^2 v$. We have
\begin{align*}
  \frac{g - h \xi_n}{q \xi_n - p}
  &= \frac{g - h\left(\frac{p}{q} + \kappa\right)}{q\left(\frac{p}{q} + \kappa\right) - p} \\
  &= \frac{g q - h p - h q \kappa}{p q + q^2 \kappa - p q} \\
  &= \frac{\epsilon - h q \kappa}{q^2 \kappa} \\
  &= \frac{\epsilon}{q^2 \kappa} - \frac{h}{q} \\
  &= \frac{\epsilon s_{n+1}^2 v}{\bar{\xi} v_n} - \frac{h}{q} \\
  &= \frac{\xi s_{n+1}^2 v_{n+1}}{s^2v} - \frac{h}{q}.
\end{align*}
Comparing this to the desired value of $\xi_{n+1} - \de$, we see that it suffices to prove that
\[
  \frac{h}{q} = \frac{\epsilon m_{n+1} s_{n+1}}{s} + \de.
\]
Since both sides lie in the interval $[0,1)$ (if $\epsilon = 1$) or $(0,1]$ (if $\epsilon = -1$), it suffices to prove that they are equal mod $1$, that is, to consider only the value of $h$ mod $q$. But $h$ mod $q$ depends only on $p$ mod $q$, which depends only on $u$ mod $s/s_n$, which we previously calculated:
\begin{align*}
  -\epsilon h &\equiv p^{-1} \mod q \\
  &\equiv \left(\frac{u}{s_{n+1}}\right)^{-1} \\
  &\equiv \left(\frac{\dfrac{a_{n+1}}{a_n/s_n} \bmod \dfrac{s}{s_n}}{s_{n+1}}\right)^{-1} \\
  &\equiv \left(\frac{a_{n+1}/s_{n+1}}{a_n/s_n}\right)^{-1} \\
  &\equiv \frac{a_n/s_n}{a_{n+1}/s_{n+1}} \mod q.
\end{align*}
Therefore
\[
  -\frac{\epsilon h}{q} \equiv \frac{\dfrac{a_n/s_n}{a_{n+1}/s_{n+1}} \bmod \dfrac{s}{s_n s_{n+1}}}{s/(s_n s_{n+1})}
  \equiv \frac{\dfrac{a_n}{a_{n+1}/s_{n+1}} \bmod \dfrac{s}{s_{n+1}}}{s/s_{n+1}} = -\frac{m_{n+1}}{s/s_{n+1}} \mod 1,
\]
as desired.
\end{proof}

We will call the continued fraction expansion of $\xih_n - \de$ appearing in the theorem the $n$th ``packet'' and denote it by $P_n$. So we have the relation
\[
  \xi - \de = \xi_0 - \de = [P_0, P_1,\ldots, P_{n-1}, \xi_n - \de]
\]
for each $n \geq 0$. We would like to deduce that $[P_0, P_1, \ldots] = \xi - \de$, but in general this is complicated by the presence of zero and negative terms at the beginnings of the packets, and hence we defer it until after proving part (c), which shows that many of its convergents are quite close to $\xi - \de$.

\begin{proof}[Proof of Theorem \ref{thm:patterns}(c)]

Let $r_n = [P_0,P_1,\ldots,P_{n-1}]$, where, for $n=0$, the empty continued fraction $[]$ is to be interpreted as $\infty$. By induction, it is enough to prove that
\[
  f_{\xi-\de}(r_n) = r_{n+1}
\]
where
\[
  f_{\xi-\de}(x) = f_{\xi}(x + \de) - \de = f(x - d + \de) + d - \de.
\]

Define the LFT's
\[
  \sigma(x) = [P_0,P_1,\ldots,P_{n-1},x]
\]
and
\[
  \tau(x) = [P_n,x].
\]
There is of course only one LFT $g$ such that
\[
  g(\sqrt{k}) = \sqrt{k}, \quad g(-\sqrt{k}) = -\sqrt{k}, \textand g(r_n) = r_{n+1};
\]
we wish to prove that $g = f_{\xi - \de}$. To this end we use the following criterion:
\begin{lem}
Let $g$ be an LFT with fixed points $p_1$ and $p_2$. For any $x \neq p_1,p_2$, the cross ratio
\[
  \lambda(g) = \cro(p_1,p_2,x,g(x))
\]
is an invariant of $g$ and (together with $p_1$ and $p_2$) determines $g$ uniquely.
\end{lem}
\begin{proof}
Without loss of generality, $p_1 = \infty$ and $p_2 = 0$. Then $g(x) = \lambda x$ for some $\lambda$, so
\[
  \cro(p_1,p_2,x,g(x)) = \cro(\infty,0,x,\lambda x) = \lambda
\]
is an invariant of $g$ and determines $g$.
\end{proof}
So it suffices to compare $\lambda(f_{\xi-\de})$ and $\lambda(g)$. For the former we have
\begin{equation}\label{eq:lf}
  \lambda(f_{\xi-\de}) = \lambda(f) = \cro(\sqrt{k}, -\sqrt{k}, \infty, d)
\end{equation}
by picking $x = \infty$. For $g$, we pick $x = r_n$ and then apply $\sigma^{-1}$ to each of the four points of the cross ratio, noting that $r_n = \sigma(\infty)$ and $r_{n+1} = \sigma(\tau(\infty)) = \sigma(\xih_n - \de)$, to get
\begin{align}
  \lambda(g) = \cro(\sqrt{k}, -\sqrt{k}, r_n, r_{n+1})
  &= \cro(\xi_n - \de, \bar{\xi}_n - \de, \infty, \xih_n - \de)
  = \cro(\xi_n, \bar{\xi}_n, \infty, \xih_n) \label{eq:lg}
\end{align}
where $\bar{\xi}_n$ is the $\QQ$-Galois conjugate of $\xi_n$. There is now no need to compute the cross ratio explicitly, since each of the four points in \eqref{eq:lf} maps to the corresponding point of \eqref{eq:lg} under the LFT
\[
  x \mapsto \frac{x s_n^2 v_n}{s^2 v} - \frac{\epsilon m_n s_n}{s}. \qedhere
\]
\end{proof}

Finally, we complete the proof of the theorem by showing that the infinite continued fraction $[P_0,P_1,\ldots]$ converges to $\xi - \de$.

\begin{proof}[Proof of Theorem \ref{thm:patterns}(a,b)]

As was previously mentioned, the quantities $\xih_n$ vary in a purely periodic manner, so there is an $\ell > 0$ such that for all $n \geq 0$,
\[
  P_{n+\ell} = P_n.
\]
This means that the set of convergents of the pattern continued fraction is a union of finitely many orbits of the LFT
\[
  \tau(x) = [P_0,P_1,\ldots,P_{\ell-1}].
\]
From the foregoing we can see that $\tau = f_{\xi-\de}^\ell$ (the two LFT's agree on $\pm \sqrt{k}$ and $\infty$) and so the orbit $\{x, \tau(x), \tau^2(x),\ldots\}$ tends to $\xi - \de$ for any rational $x$. Accordingly, we have convergence and the identity
\[
  [P_0,P_1,\ldots] = \xi - \de. \qedhere
\]
\end{proof}
\section{More on convergents and semiconvergents}
\label{sec:corollaries}

We can now give a second proof of Theorems \ref{thm:conv} and \ref{thm:semi} based on the characterization of the continued fraction found in Theorem \ref{thm:patterns}. We begin by restating the hypotheses of these theorems in terms of our parameters $s$, $v$, $m$, and $\epsilon$.

\begin{lem} \label{prop:nice}
Let $\bar{\xi} = d - \sqrt{k}$ be the Galois conjugate of $\xi$. Then
\begin{enumerate}[$(a)$]
  \item $|\bar{\xi}| < 1$ if and only if $m \geq s + \de$.
  \item $|\bar{\xi}| < 1/2$ if and only if $m \geq 2s + \de$.
\end{enumerate}
\end{lem}
\begin{proof}
Since $\sgn \bar{\xi} = -\epsilon$, it makes sense to prove the $\epsilon = 1$ and $\epsilon = -1$ cases separately. Here is a proof of (a) for $\epsilon = -1$:
\[
  \bar{\xi} < 1 \iff \sqrt{k} > d-1 \iff d^2 - k < 2d-1 \iff s^2 v < svm - 1 \iff m > s + \frac{1}{sv}.
\]
This is equivalent to $m \geq s + 1$ unless $s = v = 1$ and $m = 2$, an impossibility (since $k$ would equal $0$).

The other three cases are similar and are left to the reader.
\end{proof}

\begin{thm} \label{thm:conv2}
If $|\bar{\xi}| < 1/2$, then the pattern continued fraction is simple and the iterates $f_\xi^n(\infty)$ are convergents of $\xi$.
\end{thm}
\begin{proof}
We have the bound
\begin{align*}
  \xih_n - \de &= \frac{s_n v_n m - \epsilon m_n}{s/s_n} - \de \\
  &\geq \frac{s_n v_n m}{s/s_n} - 1 \\
  &\geq \frac{m}{s} - 1 \\
  &\geq 2 - 1 = 1.
\end{align*}
Moreover, at least one of the inequalities is strict (if $\epsilon = 1$ then $m_n < s/s_n$, and if $\epsilon = -1$ then $m/s > 2$) so $\xih_n$ exceeds $1$ and both of its continued fraction expansions have strictly positive terms. Hence the pattern continued fraction is simple and its distinguished convergents $f_\xi^n(\infty) = \de + [P_0,\ldots,P_{n-1}]$ are convergents of $\xi$.
\end{proof}

\begin{thm} \label{thm:semi2}
If $|\bar{\xi}| < 1$, then the pattern continued fraction has nonnegative terms and the iterates $f_\xi^n(\infty)$ are semiconvergents of $\xi$.
\end{thm}
\begin{proof}
We use the same method, but the bound $m \geq s + \de$ yields $\xih_n - \de > 0$ so we get a pattern continued fraction with \emph{nonnegative} terms. To obtain a simple continued fraction from this, it is necessary to eliminate the zeros using the transformation rule
\begin{equation} \label{eq:deletezero}
  [\ldots,x,0,y,\ldots] = [\ldots,x+y,\ldots].
\end{equation}
It is easy to see that this rule will compute each term of the simple continued fraction in finitely many steps unless it encounters an infinite tail of the form $[c_0,0,c_1,0,c_2,0,\ldots]$, which is impossible by the irrationality of the value of the pattern continued fraction. Moreover, it is easy to see that the two continued fractions on either side of \eqref{eq:deletezero} have the same set of semiconvergents, implying that the distinguished convergents $f_\xi^n(\infty)$ are semiconvergents of the resulting simple continued fraction.
\end{proof}

\section{Families}
In addition, we get extensive \emph{families} of continued fractions.
\begin{thm} \label{thm:family}
Fix $s$ and $\epsilon$, and let $v$ and $m$ vary within fixed congruence classes mod $s$. Then each term of the pattern continued fraction stays constant except the initial term of each packet $P_n$, which is linear in either $m$ (for odd $n$) or $vm$ (for even $n$).
\end{thm}
\begin{proof}
We have
\[
  \xih_n = \frac{s_n v_n m - \epsilon m_n}{s/s_n}.
\]
The values of $s_n$ and $m_n$ depend only on the $a_n$'s mod $s$, which in turn depend only on $v$ and $m$ mod $s$. Consequently the numerator is constant mod $s$, implying that the continued fraction expansion of $\xih_n$ is fixed except for the leading term, which is linear in $v_n m$ since $\xih_n$ is.
\end{proof}

\begin{cor}
Suppose $s$ and $\epsilon$ are fixed, and allow $v$ and $m$ to vary within fixed congruence classes mod $s$ such that $m \geq 2s + \de$ and
\[
  k = \frac{s^2 v (vm^2 + 4\epsilon)}{4}
\]
is an integer. Then each term of the continued fraction expansion of $\sqrt{k}$ is either constant, linear in $m$, or linear in $vm$, the last two cases occurring in alternation.
\end{cor}
\begin{proof}
The condition $m \geq 2s + \de$ ensures that the pattern continued fraction is simple. Thus the only alteration needed to produce the continued fraction expansion of $\sqrt{k}$ is to subtract $d$ from the first term. Since $d = svm/2$ is linear in $vm$ (and the unaltered first term, which begins the $0$th packet, is already linear in $vm$), the linearity properties are unchanged.
\end{proof}

When $\epsilon = -1$, the packets all have even length and every second term of the continued fraction is constant. This implies that the \emph{minus continued fractions} of these $\sqrt{k}$ form families, generalizing the family in Theorem Minus of \cite{RST}, which corresponds to $s=1$ in our notation.

\section{Pellian convergents}

Our final task is to determine which Pellian convergents appear in the orbit $\{f^n(\infty)\}$. By Theorem \ref{thm:easy} there will always be at least some if $R \in \ZZ$, and it is not hard to determine for which $n$ they appear.

\begin{thm}
If $k$ and $d$ are integers in the form of Proposition \ref{prop:parametrize}, the iterate $f^n(\infty)$ is Pellian if and only if the following two conditions are satisfied:
\begin{itemize}
  \item $n$ is even or $v = 1$;
  \item $a_n$ is a multiple of $s$.
\end{itemize}
\end{thm}
\begin{proof}
We use the standard fact that if we cut a continued fraction $\sqrt{k} = [c_0,c_1,\ldots]$ and compare the resulting convergent and remainder
\[
  \frac{p}{q} = [c_0,\ldots,c_{n-1}] \textand \frac{P+\sqrt{k}}{Q} = [c_n,c_{n+1},\ldots],
\]
then $p^2 - kq^2 = (-1)^n Q$ (Theorem 7.22 of \cite{NZM}, where it is to be noted that the positivity of the $c_i$ is not used). Therefore an iterate $f^n(\infty)$ is Pellian if and only if the corresponding remainder $[P_{n},P_{n+1},\ldots]$ of the pattern continued fraction is of the form $P+\sqrt{k}$ with no denominator. But this remainder is
\[
  \xi_n - \de = \frac{\xi s_n^2 v_n - \epsilon s s_n v m_n}{s^2 v} - \de,
\]
so $f^n(\infty)$ is Pellian if and only if $s_n^2 v_n = s^2 v$, which reduces to $v_n = v$ and $s_n = s$ which are respectively the two conditions in the statement of the theorem.
\end{proof}
In particular, the minimal $n$ yielding a Pellian iterate is independent of $v$ and $m$ when these remain in fixed congruence classes mod $s$, except that $n$ may jump down by a factor of $2$ when $v$ becomes $1$.

Pellian convergents outside the orbit $\{f^n(\infty)\}$, that is to say, periodicities in the simple continued fraction not reflected in the pattern continued fraction, are much trickier to study. Since we need to analyze cases where $k$ is not an integer, we first generalize the notion of a Pellian convergent to arbitrary quadratic integers.

\begin{defn}
If $\xi$ is an algebraic integer satisfying a quadratic equation $\xi^2 - t\xi - u = 0$ and $\xi > |\bar{\xi}|$, then a fraction $p/q$ ($p,q \in \ZZ$) is called \emph{Pellian} for $\xi$ if
\begin{equation} \label{eq:pellian}
  q>0, \quad \frac{p}{q} > \frac{t}{2}, \textand p^2 - t p q - u q^2 = \pm 1.
\end{equation}
\end{defn}
(The first two conditions generalize the restrictions $p>0$, $q>0$ used to filter out the redundant solutions of the ordinary Pell equation.)
\begin{prop}
If $p$ and $q$ are integers and $\xi$ a quadratic integer with $\xi > |\bar{\xi}|$, the following are equivalent:
\begin{enumerate}[$(a)$]
  \item $p/q$ is Pellian for $\xi$;
  \item $p - q\bar{\xi}$ is a unit in $\ZZ[\xi]$ and exceeds the absolute value of its conjugate;
  \item $p/q$ is a convergent in the continued fraction expansion of $\xi$ built from a number of terms that is divisible by the period length $L$.
\end{enumerate}
\end{prop}
\begin{proof}
The equivalence of (a) and (b) is purely formal: the three inequalities \eqref{eq:pellian} can be written in terms of $\alpha = p - q\bar{\xi}$ as
\[
  \alpha > \bar{\alpha}, \quad \alpha > -\bar{\alpha}, \textand \alpha\bar{\alpha} = \pm 1.
\]

To prove that (c) implies (a), we may first replace $\xi$ by $\xi - \lceil \bar{\xi} \rceil$ to assume that $-1 < \bar{\xi} < 0$. We then have $\xi > 1$ (since $\xi > \bar{\xi}$ and $|\xi \bar{\xi}| = |u| \geq 1$), and it is well known (see \cite{NZM}, Theorem 7.20) that this implies $\xi$ has a purely periodic continued fraction expansion
\[
  \xi = [\overline{c_0,\ldots,c_{L-1}}].
\]
If $p/q$ is the $nL$th convergent for some $n \geq 1$, it is easy to prove that
\begin{equation} \label{eq:qdet}
  [c_0, c_1,\ldots, c_{nL-1}, x] = \frac{p x + u q}{q x + p - t q}
\end{equation}
by comparing the images of $\xi$, $\bar{\xi}$, and $\infty$ under the LFT's on each side; using the determinant condition $p(p-tq) - uq^2 = \pm 1$ and the bounds
\[
  q > 0 \textand \frac{p}{q} \geq \lfloor \xi \rfloor = t > \frac{t}{2},
\]
we deduce that $p/q$ is Pellian.

Finally we prove that (a) implies (c). Again we may assume $\xi$ is purely periodic; the bounds $\xi > 0 > \bar{\xi} > -1$ imply that the quadratic $x^2 - t x - u$ is negative at $x = 0$ and positive at $x = -1$, yielding the bounds
\[
  t \geq u > 0.
\]
Assume that $p/q$ is Pellian with $p^2 - t p q - u q^2 = \epsilon = \pm 1$; then
\[
  g(x) = \frac{p x + u q}{q x + p - t q}
\]
is an LFT of determinant $\epsilon$. We would like to use Lemma \ref{lem:concat} to conclude that $g(x) = [c_0,\ldots,c_n,x]$ where $[c_0,\ldots,c_n]$ is one continued fraction expansion of $p/q$ and the parity of $n$ is determined by $\epsilon$. It suffices to prove the bounds
\begin{equation} \label{eq:bound}
  0 \leq p-tq \leq q,
\end{equation}
that is,
\[
  t \leq \frac{p}{q} \leq t+1,
\]
since the equality clearly can only hold when $\epsilon$ is negative or positive respectively, making $n$ even or odd respectively. The left inequality of \eqref{eq:bound} is straightforward:
\[
  p-tq = \frac{u q^2 + \epsilon}{p} \geq 0
\]
since $p$, $q$, and $u$ are positive. For the right inequality, if $p - tq > q$ then $p > (t+1)q$ and
\[
  \epsilon = p(p-tq) - uq^2 > (t-u+1)q^2 \geq 1,
\]
a contradiction. Hence
\[
  g(x) = [c_0,\ldots,c_n,x] = \frac{p x + u q}{q x + p - t q}
\]
fixes $\xi$; we obtain a continued fraction
\[
  \xi = [\overline{c_0,\ldots,c_n}],
\]
and conclude that $p/q$ consists of one or more complete periods.
\end{proof}
\begin{rem}
When $\xi = \sqrt{k}$, our proof of (c)${} \Leftrightarrow {}$(a) offers a refreshing alternative to the standard solution of the Pell equation, which goes through a convergent criterion such as Lemma \ref{lem:conv}.
\end{rem}
A by-product of our proof method is that Pellian fractions behave nicely with respect to LFT's:
\begin{lem}
Let $p_n/q_n$ denote the $n$th Pellian fraction for $\xi$, formed from $nL$ terms of its continued fraction. If $g$ is an LFT such that
\[
  g(\xi) = \xi, \quad g(\bar{\xi}) = \bar{\xi}, \textand g(\infty) = p_n/q_n,
\]
then $g^i(\infty) = p_{in}/q_{in}$.
\end{lem}
\begin{proof}
The three values given for $g$ are sufficient to identify it as the LFT in \eqref{eq:qdet}, or for general $\xi$,
\[
  g(x) = [c_0,c_1,\ldots,c_{nL-1},c_{nL}-c_0+x],
\]
which clearly takes $p_i/q_i$ to $p_{i+n}/q_{i+n}$.
\end{proof}
In particular, the $L$th convergent $p_1/q_1$ is fundamental in the sense that if an LFT fixing $\xi$ and $\bar{\xi}$ hits it, when iterating from $\infty$, then the LFT hits all Pellian convergents of $\xi$. Using a technique similar to Lemma \ref{lem:AB}, we deduce that
\[
  p_n - q_n \bar{\xi} = (p_1 - q_1 \bar{\xi})^n,
\]
so $p_1 - q_1 \bar{\xi}$ is a fundamental unit in the order $\ZZ[\xi]$.

\begin{thm}\label{thm:pell}
If $\xi = d + \sqrt{k}$, where $d,k \in \QQ$ have the form in Proposition \ref{prop:parametrize}, then the orbit $\{f_\xi^n(\infty)\}$ contains all Pellian convergents to $\xi$, except in the following cases:
\begin{enumerate}[(1)]
  \item $\xi = s\cdot \dfrac{3 + \sqrt{5}}{2}$ or $\xi = s\cdot \dfrac{5 + \sqrt{5}}{2}$, where $s$ divides some Fibonacci number $F_{2n+1}$ of odd index, using the definition
  \[
    F_0 = 0, \quad F_1 = 1, \quad F_{n+1} = F_n + F_{n-1}.
  \]
  \item $\xi = s(2 + \sqrt{2})$, where $s$ divides some ``Pell number'' $G_{2n+1}$ of odd index, using the definition
  \[
    G_0 = 0, \quad G_1 = 1, \quad G_{n+1} = 2G_n + G_{n-1}.
  \]
\end{enumerate}
\end{thm}

\begin{proof}
We begin by dealing with the case $s=1$, as larger values of $s$ simply scale the iterates of $f_\xi$ by $s$ and will be dealt with in a simple way afterwards.

Thanks to Theorem \ref{thm:family}, the pattern continued fraction expansion of $\xi$ for $s=1$ has only two possible shapes, corresponding to $\epsilon = 1$ and $\epsilon = -1$. If $\epsilon = 1$, the pattern continued fraction is
\[
  \xi = [\overline{vm, m}],
\]
which is necessarily simple, and the packets are of length $1$. Hence the orbit contains all Pellian convergents because it consists of all convergents.

If $\epsilon = -1$, then Theorem \ref{thm:patterns} instead yields the continued fraction
\[
  \xi = [vm - 1, \overline{1, m-2, 1, vm-2}]
\]
with packets of length $2$. If $m \geq 3$, this continued fraction is simple. It ordinarily has period $4$ (if $v \geq 2$) or $2$ (if $v = 1$), causing the orbit to contain the Pellian convergents, with one exception: $v=1$ and $m=3$, where $\xi = [2,\overline{1}]$ has period $1$. Here $\xi = \frac{3 + \sqrt{5}}{2}$, yielding the first exceptional case in the statement of the theorem.

If $m = 2$, the pattern continued fraction has a zero and simplifies:
\[
  \xi = [2v-1,\overline{1,0,1,2v-2}] = [2v-1,\overline{2,2v-2}].
\]
We must have $v\geq 2$ for $\xi$ to be irrational, so the last continued fraction is simple and has period $2$ (implying that the convergent $2v-\frac{1}{2} = f_\xi^2(\infty)$ is the first Pellian one) unless $v = 2$, leading to another exceptional case $\xi = 2 + \sqrt{2}$.

If $m = 1$, we must have $v \geq 5$ for $\xi$ to be real and irrational. The pattern continued fraction---with one term $-1$---is not easy to simplify, but by various means (e.g{.} comparison to $\xi-2$, which also has a pattern continued fraction corresponding to putting $1$ for $s$ and $m$, $v-4$ for $v$, and $1$ for $\epsilon$), we see that the correct simple continued fraction is
\[
  \xi = [v-1, \overline{1, v-4}].
\]
If $v \geq 6$, then the first Pellian convergent is $[v-1, 1] = v$ which is also the second iterate of $f_\xi$. But for $v = 5$ the period becomes $1$ and the first convergent $v-1$ is also Pellian, leading to the final exceptional case $\xi = \frac{5 + \sqrt{5}}{2}$.

For general $s$, consider
\[
  \frac{\xi}{s} = \frac{vm + \sqrt{v(vm^2 + 4\epsilon)}}{2}.
\]
If $\xi/s$ is not one of the three exceptions to the above analysis when $s=1$, then the LFT $f_{\xi/s}$ finds the fundamental unit in the order $\ZZ[\xi/s]$, of which $\ZZ[\xi]$ is a suborder. Consequently, $f_\xi$ picks up all Pellian convergents to $\xi$ in this case.

If $\xi/s$ is $\frac{3+\sqrt{5}}{2}$ or $\frac{5+\sqrt{5}}{2}$, then we are dealing with the order $\ZZ[\xi/s] = \ZZ[\phi]$, where $\phi = \frac{1+\sqrt{5}}{2}$ is the golden ratio. In either case, the fundamental unit is $\phi$ but $f_{\xi/s}$ only picks up $\phi^2$ and its powers. Therefore, $f_\xi$ misses a Pellian convergent of $\xi$ if and only if the order $\ZZ[\xi] = \ZZ[s\phi]$ contains some odd power $\phi^{2n+1}$ of $\phi$. In view of the identity
\[
  \phi^n = F_n\phi + F_{n-1},
\]
this holds if and only if $s|F_{2n+1}$ for some $n$.

If $\xi/s = 2 + \sqrt{2}$, the proof is exactly analogous: the order $\ZZ[\xi/s] = \ZZ[\sqrt{2}]$ has fundamental unit $\alpha = 1 + \sqrt{2}$, but the LFT $f_{\xi/s}$ only finds its square. Using the identity
\[
  \alpha^n = G_n\alpha + G_{n-1},
\]
we find that $f_\xi$ misses a Pellian convergent if and only if $s|G_{2n+1}$ for some $n$.
\end{proof}
\begin{rem}
In view of the identities
\[
  F_{2n+1} = F_n^2 + F_{n-1}^2 \textand G_{2n+1} = G_n^2 + G_{n-1}^2,
\]
any exceptional value of $s$ divides a sum of two coprime squares and thus equals a product of primes congruent to $1$ mod $4$, with an optional factor of $2$. Not all such $s$ divide some $F_{2n+1}$ or $G_{2n+1}$, however ($s=29$ fails in the Fibonacci case, and $s=17$ fails in the Pell number case).
\end{rem}
\begin{cor}\label{cor:allpell}
If $d$ and $k$ are integers with $4d^2/(k-d^2) \in \ZZ$, then the orbit $\{f^n(\infty)\}$ contains all Pellian convergents except when $(k,d)$ is a pair of one of the forms
\begin{itemize}
  \item $k = 5s^2$, $d = 3s$ or $5s$, where $2s|F_{2n+1}$;
  \item $k = 2s^2$, $d = 2s$, where $s|G_{2n+1}$.
\end{itemize}
\end{cor}
(In the exceptional cases, the choices $d = 2s$ and $d = s$, respectively, may be used instead if an LFT hitting all Pellian convergents is desired.)

Since the exceptions occur only when the LFT jumps to a solution of the positive Pell equation, skipping over a solution of the negative one, we have the following simple corollary.
\begin{cor}
If $d$ and $k$ are integers with $4d^2/(k-d^2) \in \ZZ$, then the orbit $\{f^n(\infty)\}$ contains all convergents $p/q$ satisfying the positive Pell equation $p^2 - k q^2 = 1$.
\end{cor}
Finally, we get results about the solvability of the negative Pell equation.
\begin{cor}
If $k$ is an integer such that $4d^2/(k-d^2)$ is a \emph{negative} integer for some $d$, then the negative Pell equation $p^2 - k q^2 = -1$ has no solutions, unless $(k,d)$ is one of the exceptions in Corollary \ref{cor:allpell}.
\end{cor}
\begin{proof}
The positivity of $d^2 - k$ is equivalent to $\epsilon = -1$, implying that all the iterates of $f$ lie above $\sqrt{k}$ and thus that any Pellians among them satisfy the positive Pell equation $p^2 - k q^2 = 1$.
\end{proof}

\section{Open questions}
\label{sec:concl}

A direction of generalization that immediately suggests itself is to iterate $f$ on initial values other than $\infty$; however, this case is almost entirely solved by the foregoing theorems. If $R$ is \emph{not} an integer, the proof of Theorem \ref{thm:easy} shows that no orbit of $f$ can contain two convergents in corresponding places within the continued fraction period; thus every orbit contains finitely many convergents, at most one of which is Pellian. If $R$ is an integer, the orbit (unless it coincides with the orbit of $\infty$) misses all the Pellian convergents but could possibly include an infinite family of convergents lying in corresponding places with respect to the period. The question then arises whether, for some $k$ and $d$, a single orbit might contain two or more convergents per period.

As $m \to \infty$, the iterates $p/q \in \{f^n(\infty)\}$ have ``Pellian error'' $p^2 - k q^2$ bounded by $s^2 v$, while the Pellian errors of all other convergents appear to tend to $\infty$. Is there a theorem in the spirit of Lemma \ref{lem:conv} that, if $|\sqrt{k} - d|$ is sufficiently small, then any fraction whose Pellian error is at most $s^2 v$ is an iterate of $f$?

Since many $k$ do not have any integer $d$ making $R$ an integer ($k = 19$ is the smallest; their density is doubtless $1$), it is natural, from the point of view of computing continued fractions and Pell equation solutions, to consider non-integral $d$. By Theorem \ref{thm:easy}, iterating $f$ on $\infty$ eventually yields a Pellian convergent, but is it the \emph{first} Pellian convergent if, for instance, we take $d$ to be the first convergent of $\sqrt{k}$ for which $R \in \ZZ$ holds? Also, we can seek analogues of the $|d-\sqrt{k}|$ conditions for the iterates to all be convergents or semiconvergents. Most intriguingly, do the resulting continued fractions fit into families, as in Theorem \ref{thm:family}, and can their terms be described by explicit formulas similar to Theorem \ref{thm:patterns}? Many of the same questions can be asked if $k$ is a non-integer, thus entering the realm of approximating arbitrary quadratic surds $\frac{p+\sqrt{k}}{q}$.
 
Finally, our work says nothing about the structure of the continued fraction expansion of $\sqrt{k}$ when $k$ is close to $d^2$ yet $R$ is not an integer. Although the orbit of $\infty$ under $f$ necessarily contains finitely many continued fraction convergents, it can contain arbitrarily many, a proof of which is suggested by the following example:
\[
  \sqrt{10^8+3} = [ 10000, 6666, 1, 2, 2221, 1, 8, 740, 1, 1, 1, 2, 2, 1, 246, 4, 1, 3, 4, 82, \ldots]
\]
This looks like the concatenation of the continued fraction expansions of certain numbers $10000$, $\frac{20000}{3}$, $\frac{19997}{9}$, $\frac{19997}{27},\ldots$ decreasing approximately by powers of $3$, and one can calculate that truncating at these spots indeed yields the iterates of $f = f_{10000}$. We may seek a formula analogous to that of Theorem \ref{thm:patterns} that expresses each packet as the continued fraction expansion of a rational number obtained from some recursion related to $f$. If this is continued forever, the resulting continued fraction is almost surely non-simple and non-periodic, but we can still ask whether it converges and whether the Pellian fractions are hidden among its convergents or semiconvergents. We wonder whether this last line of investigation extends to cube roots and to transcendental functions, where function-termed continued fraction expansions have long been known (e.g{.} $\tanh(1/x) = [0,x,3x,5x,\ldots]$), but few attempts have been made to convert such continued fractions into ones with integer terms in the case that $x$ is large and rational.

\bibliography{contfrac.bib}
\bibliographystyle{plain}

\end{document}